\documentclass[11pt]{amsart}

\usepackage[latin1]{inputenc}

\usepackage{amssymb,amsmath,amstext,amsthm,amsfonts}
\usepackage{amsfonts}
\usepackage{amsmath}
\usepackage{amssymb}
\usepackage{amsthm}
\usepackage{graphics}
\usepackage{epsfig}

\newcommand{\bdm}{\begin{displaymath}}
\newcommand{\edm}{\end{displaymath}}

\theoremstyle{definition}

\newtheorem{lem}{Lemma}
\newtheorem{thm}{Theorem}
\newtheorem{defn}{Definition}
\newtheorem{prop}{Proposition}
\newtheorem{rem}{Remark}
\newtheorem{eg}{Example}

\title[]{Coloring Factors of Substitutive Infinite Words}

\author{Andr\'{e} Bernardino}
\address{ Departamento de Matem\'{a}tica, Universidade da Beira Interior\\
 6200-001 Covilh\~{a}, Portugal}
\email{and\_bernardino@hotmail.com}

\author{Rui Pacheco}
\address{ Departamento de Matem\'{a}tica, Universidade da Beira Interior\\
 6200-001 Covilh\~{a}, Portugal}
\email{rpacheco@ubi.pt}

\author{Manuel Silva}
\address{ Departamento de Matem\'{a}tica, Universidade Nova de Lisboa\\ 2829-516 Caparica, Portugal}
\email{mnas@fct.unl.pt}

\begin{document}

\begin{abstract}
In this paper, we consider  infinite words that arise as fixed points of primitive substitutions on a finite alphabet and finite colorings of their factors. Any such infinite word exhibits a ``hierarchal structure" that  will allow us to define, under the additional  condition of \emph{strong recognizability},
 certain remarkable finite colorings of its factors. In particular we generalize two combinatorial results by Justin and Pirillo concerning arbitrarily large monochromatic $k$-powers occurring in infinite words; in view of a recent paper by de Luca, Pribavkina and Zamboni, we will give new examples of classes of infinite words $\mathbf{u}$ and finite colorings that do not allow infinite monochromatic factorizations $\mathbf{u}=\mathbf{u}_1\mathbf{u}_2\mathbf{u}_3\ldots$.
\end{abstract}
\maketitle
\section{Introduction}
 Let $\mathcal{A}$ be a finite alphabet and $\mathcal{A}^*$ be the set of all finite words over $\mathcal{A}$.
In this paper we consider  infinite words that arise as fixed points of primitive substitutions on $\mathcal{A}$ and finite colorings of $\mathcal{A}^*$. Any such infinite word exhibits a ``hierarchal structure",  induced by the underlying substitution, that  will allow us to define, under the additional  condition of \emph{strong recognizability}  (in the spirit of the recognizability conditions
introduced by B. Host \cite{Host} and B. Moss\'{e \cite{Mo}}),
 certain remarkable finite colorings of its factors.

For a two letter alphabet $\{0,1\}$,  J. Justin and G. Pirillo \cite{JP} constructed a finite coloring $c$ of $\{0,1\}^*$
with respect to which the Thue-Morse word $\mathbf{u}^T$ avoids  \emph{uniform monochromatic $3$-powers}, that is  any word  of the form $\mathbf{w}_1\mathbf{w}_2\mathbf{w}_3$ occurring in $\mathbf{u}^T$, where the $\mathbf{w}_i$ are of equal length, satisfies $c(\mathbf{w}_i)\neq c(\mathbf{w}_j)$ for some $i$ and $j$.
The  Thue-Morse word is a fixed point of the substitution $\zeta_T$ defined by $\zeta_T(0)=01$ and $\zeta_T(1)=10$.
In order to define the coloring $c$, the authors made implicit use of the recognizability properties of this substitution.
In Section \ref{JP1} we will give a generalization of their result: for a large class of substitutions of constant length on a finite alphabet $\mathcal{A}$, there always exist finite colorings of $\mathcal{A}^*$ with respect to which any fixed word avoids arbitrarily large uniform monochromatic $k$-powers.

Once this established, it is natural to consider next arbitrarily large  monochromatic $k$-powers $\mathbf{w}_1\mathbf{w}_2\ldots\mathbf{w}_k$ with \emph{uniformly bounded gaps}, that is  the lengths of the factors $\mathbf{w}_i$ are bounded by some $p>0$ and $p$ is independent of $k$.
J. Justin and G. Pirillo \cite{JP} showed that there exists a 3-coloring of $\{0,1\}^*$ with respect to which the
 fixed point $\mathbf{u}^J$ of the substitution $\zeta_J$ defined by $\zeta_J(0)=00001$ and $\zeta_J(1)=11110$ avoids arbitrarily large  monochromatic $k$-powers with uniformly bounded gaps.
The main ingredients in their argument are the following. Firstly,
 $\mathbf{u}^J$ avoids \emph{abelian $5$-powers} \cite{Ju}, that is, there does not exist any finite word of the form $\mathbf{w}=\mathbf{w}_1\ldots \mathbf{w}_5$ occurring in $\mathbf{u}^J$ where any two of  the factors $\mathbf{w}_i$ are permutations of each other. Secondly, any very long factor of  $\mathbf{u}^J$ contains ``about equally many $0$'s and $1$'s".
 In Section \ref{JP2} we will be able to show that, more generally, if an infinite word $\mathbf{u}$ over a finite alphabet satisfies
 \begin{enumerate}
   \item $\mathbf{u}$ avoids arbitrarily large abelian $k$-powers;
   \item  each factor of $\mathbf{u}$ occurs in $\mathbf{u}$ with a \emph{uniform frequency};
 \end{enumerate}
 then  $\mathbf{u}$ avoids arbitrarily large  monochromatic $k$-powers with uniformly bounded gaps. Again, a large class of substitutions gives rise to infinite words satisfying the above properties.
 The property of uniform frequencies is close related to the property of unique ergodicity of the dynamical system associated to the action of the shift map on $\mathbf{u}$. It is well known that such dynamical system is uniquely ergodic if $\mathbf{u}$ is the fixed word of a primitive substitution \cite{Q}.

In \cite{LZ,LPZ} the authors discussed the following question: \emph{given an infinite word $\mathbf{u}$ over  a finite alphabet, does there exist a finite coloring of its finite factors which avoids monochromatic factorizations of $\mathbf{u}$}? They showed that this question, which is ultimately motivated by a result of Schutzenberger \cite{S}, has a positive answer for all non-\emph{uniformly recurrent} words and for various classes of uniformly recurrent words.
 V. Salo and I. T\"{o}rm\"{a} \cite{ST} showed  that any aperiodic linearly recurrent word $\mathbf{u}$ admits a coloring of its factors that avoids monochromatic factorizations of $\mathbf{u}$ into factors of increasing lengths.
We will prove, in Section \ref{LZ}, that the question has also a positive answer  for a wide class of  infinite words arising as fixed points of primitive substitutions and we shall be able to construct examples that do not fit in those classes of infinite words considered in \cite{LZ,LPZ}. Again, the ``hierarchal structure"  of these words and the strong recognizability condition will play a fundamental role in the construction of our colorings.

\vspace{.10in}

\textbf{Acknowledgements.} This work was  supported by the Center of Mathematics and
Applications of Universidade da Beira Interior   through the project
UID/MAT/00212/2013  and the Center of Mathematics and Applications of Universidade  Nova de Lisboa through the   project        UID/MAT/00297/2013.

\section{Preliminaries}
\subsection{Monochromatic factorizations}  The binary operation obtained by concatenation of two finite words endows $\mathcal{A}^*$ with a monoid structure (the identity is the empty word $\emptyset$). For each $\alpha\in \mathcal{A}$, the $\alpha$-\emph{length} of a finite word $\mathbf{w}\in\mathcal{A}^*$, which we denote by $|\mathbf{w}|_\alpha$, is the number of occurrences of $\alpha$ in $\mathbf{w}$. The \emph{length} of $\mathbf{w}$ is the sum $|\mathbf{w}|=\sum_{\alpha\in\mathcal{A}}|\mathbf{w}|_\alpha$ of all its $\alpha$-lengths. If a copy of $\mathbf{w}$ occurs in a word $\mathbf{u}$, we say that $\mathbf{w}$ is a \emph{factor} of $\mathbf{u}$.

 Given an infinite word $\mathbf{u}=\alpha_1\alpha_2\alpha_3\ldots$ over $\mathcal{A}$,  $\mathcal{L}(\textbf{u})$ stands for the \emph{language} of $\mathbf{u}$, that is the set of all nonempty finite factors of $\mathbf{u}$. For each $i$ and $j$ satisfying $1\leq i\leq j$, we write $$\mathbf{u}_{[i,j]}=\alpha_i\alpha_{i+1}\ldots \alpha_j\in \mathcal{L}(\textbf{u}).$$ Let $c:\mathcal{L}(\textbf{u})\to \{1,\ldots,r\}$ be a finite coloring of $\mathcal{L}(\textbf{u})$.
\begin{defn}\cite{JP}
   A finite word $\mathbf{w}\in \mathcal{L}(\mathbf{u})$ is a \emph{monochromatic $k$-power} if there exists a factorization $\mathbf{w}=\mathbf{w}_1\mathbf{w}_2\ldots \mathbf{w}_k$, with $\mathbf{w}_i\in \mathcal{L}(\mathbf{u})$, such that $c(\mathbf{w}_i)=c(\mathbf{w}_j)$, for all $i$ and $j$. A monochromatic $k$-power is \emph{uniform} if $|\mathbf{w}_i|=|\mathbf{w}_j|$ for all $i$ and $j$.
 A monochromatic $k$-power  has  \emph{gaps bounded} by $p$ if $|\mathbf{w}_i|< p$ for all $i$.
  \end{defn}



\subsection{Substitutive words} Next we recall some fundamental facts concerning substitutive words. For further details we refer the reader to \cite{Q}.
A \emph{substitution} $\zeta$ of $\mathcal{A}$ is a map from $\mathcal{A}$ to $\mathcal{A}^*$ which associates the letter $\alpha$ to some word $\zeta(\alpha)\in\mathcal{A}^*$. This induces a morphism of the monoid $\mathcal{A}^*$  by putting $\zeta(\emptyset)=\emptyset$ and $$\zeta(\alpha_1\alpha_2\ldots\alpha_k)=\zeta(\alpha_1)\zeta(\alpha_2)\ldots \zeta(\alpha_k).$$
For each $\alpha\in\mathcal{A}$, let $l_\alpha=|\zeta(\alpha)|$ so that
  $$\zeta(\alpha)=\zeta(\alpha)_1\zeta(\alpha)_2\ldots \zeta(\alpha)_{l_{\alpha}}.$$
The substitution has \emph{constant length} $l$ if $l=l_\alpha$ for all $\alpha\in\mathcal{A}$.
The substitution $\zeta$ is said to be \emph{primitive} if there exists $n>0$ such that, for every $\alpha,\beta\in \mathcal{A}$, $\beta$ occurs in $\zeta^n(\alpha)$ (that is, $n$ can be chosen independent of $\alpha$ and $\beta$). Henceforth we will assume that  $\alpha=\zeta(\alpha)_1$ for some $\alpha\in\mathcal{A}$.

Consider the space of \emph{$\zeta$-substitutive infinite words}
$X_\zeta=\{\mathbf{u}\in \mathcal{A}^\mathbb{N}\colon\,\, \mathcal{L}(\mathbf{u})\subseteq \mathcal{L}_\zeta\}$, where
$$\mathcal{L}_\zeta=\!\!\!\!\bigcup_{n\geq 0,\alpha\in\mathcal{A}}\!\!\{\mbox{factors of $\zeta^n(\alpha)$}\},$$
and endow $\mathcal{A}^\mathbb{N}$ with the metric $d$ defined by: given  $\mathbf{u}=\alpha_1\alpha_2\ldots$ and $\mathbf{v}=\beta_1\beta_2\ldots$ in $\mathcal{A}^\mathbb{N}$,  $d(\mathbf{u},\mathbf{v})=0$ if $\mathbf{u}=\mathbf{v}$ and $d(\mathbf{u},\mathbf{v})=1/N$ if $N$ is the smallest positive integer for which $\alpha_N\neq \beta_N$. Denote by $\sigma$ the \emph{shift map} in $\mathcal{A}^\mathbb{N}$, $$\sigma(\alpha_1\alpha_2\alpha_3\ldots)=\alpha_2\alpha_3\ldots,$$ which preserves $X_{\zeta}$.
The substitution $\zeta$ also induces a map $\zeta:X_\zeta\to X_\zeta$.
\begin{prop}\cite{Q}\label{primitive}
Suppose that $\zeta$ is a primitive substitution. We have:
 \begin{enumerate}
  \item $\zeta$ admits a fixed point: $\zeta(\mathbf{u})=\mathbf{u}$, for some $\mathbf{u}\in X_\zeta$.
  \item The fixed word $\mathbf{u}$ is \emph{uniformly recurrent}, that is, for each factor $\mathbf{w}$ of $\mathbf{u}$ there exists $R>0$ such that in any other factor of $\mathbf{u}$ of length $R$ there is at least one occurrence of $\mathbf{w}$.

   \item The system $(X_\zeta,\sigma)$ is \emph{minimal}, that is, $X_{\zeta}$ is the closure of any orbit under $\sigma$: $X_\zeta=\overline{\mathcal{O}(\mathbf{v})}$ for all $\mathbf{v}\in X_\zeta$. In particular $\mathcal{L}_\zeta=\mathcal{L}(\mathbf{v})$ for all $\mathbf{v}\in X_\zeta$.
       \item Each factor of $\mathbf{u}$ occurs in $\mathbf{u}$ with a \emph{uniform frequency}. More precisely, denoting, for each finite interval $I$,  the number of occurrences of the factor $\mathbf{w}$ in  $\mathbf{u}_I$ by $L_{\mathbf{w}}(\mathbf{u}_I)$, then
$$\lim_{N\to  \infty}\frac{L_{\mathbf{w}}(\mathbf{u}_{[k,k+N]})}{N}$$
converges uniformly in $k$ to some number $d_\mathbf{w}\geq 0$.
\end{enumerate}

\end{prop}
In particular, if $\zeta$ is primitive, for each letter $\alpha\in\mathcal{A}$, the limit
$$\lim_{|\mathbf{w}|\to \infty}\frac{|\mathbf{w}|_\alpha}{|\mathbf{w}|}$$
exists and does not depend on the factors $\mathbf{w}\in \mathcal{L}(\mathbf{u})$.

\subsection{Recognizability}

 Let $\zeta:\mathcal{A}\to\mathcal{A}^*$ be a substitution and $\mathbf{u}$ a fixed point. For each $k\geq 1$, let $$E_k=\{1\}\cup\{|\zeta^k(\mathbf{u}_{[1,p]})|+1\colon\,p\geq 1\}$$
be  the set of \emph{$k$-cutting bars} of $\zeta$. Of course, the $k$-cutting bars of $\zeta$ are precisely the  $1$-cutting bars of $\zeta^k$.
Given an integer interval $I=[i^-,i^+]$, we denote by $m(I)$ the least $1$-cutting bar of $\zeta$ greater than or equal to $i^-$ and by $M(I)$ the largest $1$-cutting bar  of $\zeta$ less than or equal to $i^+$.
Let $p_I$ and $q_I$ be defined by
$$m(I)=|\zeta(\mathbf{u}_{[1,p_I]})|+1,\quad M(I)=|\zeta(\mathbf{u}_{[1,q_I]})|+1.$$
 We say that the interval $I$ is  \emph{$k$-fitted} if $i^-$ and $i^++1$ are $k$-cutting bars. Of course, if $I$ is $k$-fitted, then: $I$ is also $k'$-fitted for all $k'<k$;
$i^-=m(I)$ and $i^+=M(I)-1$;
$\mathbf{u}_I=\zeta(\mathbf{u}_{[p_I+1,q_I]}).$
For a general integer interval $I$, we also define
$$\mathbf{u}_{I,\zeta}^{\mathrm{pref}}=\alpha_{i^-}\ldots \alpha_{m(I)-1},\quad \mathbf{u}_{I,\zeta}^{\mathrm{suf}}=\alpha_{M(I)}\ldots \alpha_{i^+}.$$
When it is clear in the context which substitution we are referring to, then we simply denote $\mathbf{u}_{I,\zeta}^{\mathrm{pref}}=\mathbf{u}_I^{\mathrm{pref}}$ and   $\mathbf{u}_{I,\zeta}^{\mathrm{suf}}= \mathbf{u}_I^{\mathrm{suf}}$.


\begin{defn}\label{strong}
The substitution $\zeta$ is \emph{strongly recognizable} if there exists an integer $K$ such that: if $I$ is a $1$-fitted interval, $|I|>K$ and  $\mathbf{u}_{I}=\mathbf{u}_{J}$, then the integer interval $J$ is also $1$-fitted and $$\mathbf{u}_{[p_I+1,q_I]}=\mathbf{u}_{[p_J+1,q_J]}.$$
 The smallest integer $K$ enjoying this property is called the \emph{(strong) recognizability index} of $\zeta$.
Given a strongly recognizable substitution $\zeta$, if $\mathbf{w}\in\mathcal{L}(\mathbf{u})$, $|\mathbf{w}|>K$ and $\mathbf{w}=\mathbf{u}_I$ for some $1$-fitted interval $I$, then  we say that $\mathbf{w}$ is $1$-fitted and we denote $\zeta^{-1}(\mathbf{w})=\mathbf{u}_{[p_I+1,q_I]}$ (it is clear that this definition does not depend on the interval $I$).
\end{defn}

\begin{eg}
The Thue-Morse substitution  $\zeta_T$ on $\mathcal{A}=\{0,1\}$ defined by $\zeta_T(0)=01$ and $\zeta_T(1)=10$ is strongly recognizable. As a matter of fact,
if $|I|=4$, $I$ is $1$-fitted if, and only if, $\mathbf{u}_I$ coincides with one of the following words: $1001$, $0110$, $1010$, $0101$. On the other hand, $\zeta_T$ has constant length and is one-to-one on letters; consequently, if $I$ is $1$ fitted, $|I|\geq 4 $ and $\mathbf{u}_I=\mathbf{u}_J$, then $J$ is also $1$-fitted and $\mathbf{w}=\mathbf{u}_I=\mathbf{u}_J$ can be unequivocally ``desubstituted", that is, $\mathbf{u}_{[p_I+1,q_I]}=\mathbf{u}_{[p_J+1,q_J]}.$
\end{eg}

\begin{lem}\label{presuf}
Let $\zeta:\mathcal{A}\to\mathcal{A}^*$ be a strongly recognizable substitution with recognizability index $K$. Let $L=\max\{|\zeta(\alpha)|\colon\,\alpha\in \mathcal{A}\}$.
  Suppose that $\mathbf{u}_{I}=\mathbf{u}_{J}$, with $|I|=|J|>2L+K$. Then
  \begin{enumerate}
    \item  $\mathbf{u}_I^{\mathrm{pref}}=\mathbf{u}_J^{\mathrm{pref}}$
    \item $\mathbf{u}_I^{\mathrm{suf}}=\mathbf{u}_J^{\mathrm{suf}}$
    \item $\mathbf{u}_{[m(I),M(I)-1]}=\mathbf{u}_{[m(J),M(J)-1]}$
  \end{enumerate}
\end{lem}
\begin{proof}
  For $I=[i^-,i^+]$ and $J=[j^-,j^+]$, take   $n=j^-+m(I)-i^-$.
  Since  $|I|>2L+K$, it is clear that $[m(I),M(I)-1]$ has length $r_I=M(I)-m(I)$ larger than $K$.
Observe that  $[n,n+r_I-1]\subset J$ and write
  \begin{equation*}\label{decompo}
  \mathbf{u}_{I}=\mathbf{u}_{[i^-,m(I)-1]}\mathbf{u}_{[m(I),M(I)-1]}\mathbf{u}_{[M(I),i^+]},\,\,\,\,\mathbf{u}_{J}=\mathbf{u}_{[j^-,n-1]}\mathbf{u}_{[n,n+r_I-1]}\mathbf{u}_{[n+r_I,j^+]}.
  \end{equation*}
  From the definition of  $n$, we have $\mbox{$ |\mathbf{u}_{[i^-,m(I)-1]}|=|\mathbf{u}_{[j^-,n-1]}| $}.$
 Hence,  since
$\mathbf{u}_{I}=\mathbf{u}_{J}$, we get
$$\mathbf{u}_{[i^-,m(I)-1]} =\mathbf{u}_{[j^-,n-1]},\quad \mathbf{u}_{[m(I),M(I)-1]}=\mathbf{u}_{[n,n+r_I-1]},\quad \mathbf{u}_{[M(I),i^+]}=\mathbf{u}_{[n+r_I,j^+]}.$$ Consequently, by recognizability, the interval $[n,n+r_I-1]$ is  $1$-fitted. Moreover, we must have $n=m(J)$. Otherwise, by definition of $m(J)$, we would have $n>m(J)$; but in this case, again by recognizability and reversing the rules of $I$ and $J$, $n'=i^-+m(J)-j^-$ would be a  $1$-cutting bar satisfying $i^-\leq n' <m(I)$, which is in contradiction with the definition of $m(I)$.  Hence $n=m(J)$, $\mathbf{u}_J^{\mathrm{pref}}=\mathbf{u}_{[j^-,n-1]}$ and we conclude that $\mathbf{u}_I^{\mathrm{pref}}=\mathbf{u}_J^{\mathrm{pref}}$. The remaining assertions of the lemma follow as a consequence of this.
  \end{proof}

Given  a strongly recognizable substitution $\zeta$ with recognizability index $K$ and a $2$-fitted interval $I$ (equivalently, $I$ is $1$-fitted with respect to $\zeta^2$), assume that  $|I|>LK$, where $L=\max\{|\zeta(\alpha)|\colon\,\alpha\in \mathcal{A}\}$.
Suppose that $\mathbf{u}_I=\mathbf{u}_J$. By recognizability, the interval $J$ is $1$-fitted and $\mathbf{u}_{[p_I+1,q_I]} =\mathbf{u}_{[p_J+1,q_J]}$. Set $I'=[p_I+1,q_I]$ and $J'= [p_J+1,q_J]$. Since $I$ is $2$-fitted, $I'$ is $1$-fitted. On the other hand, $|I'|>|I|/L>K$. Hence, again by recognizability, $J'$ is $1$-fitted (consequently, $J$ is $2$-fitted) and $\mathbf{u}_{[p_{I'}+1,q_{I'}]} =\mathbf{u}_{[p_{J'}+1,q_{J'}]}$. This proves that $\zeta^2$ is strongly recognizable. More generally, this argument can be extended in order to prove by induction the following.
\begin{prop}\label{pcuttingbars}
  Suppose that $\zeta$ is strongly recognizable.  Then $\zeta^k$ is strongly recognizable for each integer $k>1$.
\end{prop}

We say that a substitution is \emph{admissible} if it is primitive and strongly recognizable.

\begin{prop}\label{aperiodic}
  If $\zeta$ is admissible, any infinite word $\mathbf{v}$ in $X_\zeta$ is not shift periodic, that is there does not exist a finite word $\mathbf{w}$ with $\mathbf{v}=\mathbf{w}^\infty=\mathbf{w}\mathbf{w}\ldots$.
\end{prop}
\begin{proof}
  The argument is standard. Let $\mathbf{u}=\alpha_1\alpha_2\alpha_3\ldots$ be a fixed point of $\zeta$ and suppose that $\mathbf{v}\in X_\zeta=\overline{\mathcal{O}(\mathbf{u})}$ is periodic.
  Since $\mathbf{v}$ is periodic, there exists $n_0$ so that $\sigma^{n_0}(\mathbf{v})=\mathbf{v}$. By primitivity, we can choose $n$ satisfying $|\zeta^n(\alpha)|>2n_0$ for all $\alpha\in \mathcal{A}$. By minimality, there exists an arbitrarily large interval $I=[i^-,i^+]$ so that $\mathbf{v}=\mathbf{u}_I\tilde{\mathbf{v}}$, where $\tilde{\mathbf{v}}$ is an infinite factor of $\mathbf{v}$. Assume that $|I|$ is much larger than the recognizability index of $\zeta^n$.   Write
  \begin{align*}
\mathbf{u}_I&=\mathbf{u}_{I,\zeta^n}^{\mathrm{pref}}\zeta^n(\alpha_{p_I+1})\zeta^n(\alpha_{p_I+2})\ldots \zeta^n(\alpha_{q_I})\mathbf{u}_{I,\zeta^n}^{\mathrm{suf}},
 \end{align*}
 where $p_I$ and $q_I$ are defined with respect to the substitution $\zeta^n$.
 We consider now two cases. Firstly, if $n_0\leq  |\mathbf{u}_{I,\zeta^n}^{\mathrm{pref}}|$, then, since $\sigma^{n_0}(\mathbf{v})=\mathbf{v}$, we have
 $$\mathbf{u}_{I,\zeta^n}^{\mathrm{pref}}\zeta^n(\alpha_{p_I+1})\zeta^n(\alpha_{p_I+2})\ldots =\mathbf{w}\zeta^n(\alpha_{p_I+1})\zeta^n(\alpha_{p_I+2})\ldots ,$$
where $\mathbf{w}$ is the prefix of $\mathbf{u}_{I,\zeta^n}^{\mathrm{pref}}$ with $|\mathbf{w}|=|\mathbf{u}_{I,\zeta^n}^{\mathrm{pref}}|-n_0<|\mathbf{u}_{I,\zeta^n}^{\mathrm{pref}}|$, which contradicts the recognizability property of Lemma \ref{presuf}. Secondly, if   $n_0> |\mathbf{u}_{I,\zeta^n}^{\mathrm{pref}}|$, then we can write, since $\sigma^{n_0}(\mathbf{v})=\mathbf{v}$,
 $$\mathbf{u}_{I,\zeta^n}^{\mathrm{pref}}\zeta^n(\alpha_{p_I+1})\zeta^n(\alpha_{p_I+2})\ldots =\mathbf{w}\zeta^n(\alpha_{p_I+2})\zeta^n(\alpha_{p_I+3})\ldots ,$$
 where $\mathbf{w}$ is the sufix of $\zeta^n(\alpha_{p_I+1})$ with $$|\mathbf{w}|= |\zeta^n(\alpha_{p_I+1})| -(n_0-|\mathbf{u}_{I,\zeta^n}^{\mathrm{pref}}|).$$
 Since $|\zeta^n(\alpha)|>2n_0$ for all $\alpha\in\mathcal{A}$, we see from this that $|\mathbf{w}|>n_0>|\mathbf{u}_{I,\zeta^n}^{\mathrm{pref}}|$, which contradicts the recognizability property of Lemma \ref{presuf}. We conclude that $\mathbf{v}$ cannot be periodic.
 \end{proof}

\section{Avoiding arbitrarily large uniform monochromatic $k$-powers}\label{JP1}

 J. Justin and G. Pirillo in \cite{JP} constructed a finite coloring of the factors of the Thue-Morse word which avoids   uniform monochromatic $3$-powers. The  Thue-Morse word is a fixed point of the admissible substitution $\zeta_T$ on $\mathcal{A}=\{0,1\}$ defined by $\zeta_T(0)=01$ and $\zeta_T(1)=10$.  Next we generalize this construction to a large class of substitutions $\zeta$ on a finite alphabet.

\begin{thm}\label{unifo}
Let $\zeta$ be an admissible substitution of constant length $L$ on a two-letter alphabet $\mathcal{A}$ with recognizability index $K$.    Let $\mathbf{u}=\alpha_1\alpha_2\ldots$  be a fixed point of $\zeta$.
Then there exists a finite coloring   of $\mathcal{L}(\mathbf{u})$ avoiding uniform monochromatic $k$-powers for all $k$ sufficiently large.
\end{thm}
\begin{proof}
We start by observing that
an admissible substitution $\zeta$  on a two-letter alphabet $\mathcal{A}=\{0,1\}$  is always one-to-one on letters (otherwise the fixed word would be periodic, in contradiction with Proposition \ref{aperiodic}). Hence, for some $r>0$, we have $\zeta(0)_r\neq \zeta(1)_r$, which means that the substitutions $\zeta_{r}^{\mathrm{pref}}$ and $\zeta_{r}^{\mathrm{suf}}$ on $\mathcal{A}$ defined by
\begin{equation}\label{subpre}
\zeta^{\mathrm{pref}}_{r}(\alpha)=\zeta(\alpha)_1\zeta(\alpha)_2\ldots\zeta(\alpha)_r,\quad \zeta^{\mathrm{suf}}_{r}(\alpha)=\zeta(\alpha)_r\zeta(\alpha)_{r+1}\ldots\zeta(\alpha)_{L},
\end{equation}
 with $\alpha\in\{0,1\}$, are both one-to-one on letters.

Define a finite coloring $c$ of $\mathcal{L}(\mathbf{u})$ as follows. For  $\mathbf{w}\in\mathcal{L}(\mathbf{u})$:
\begin{enumerate}
\item if $|\mathbf{w}|\leq 2L+K$, then  $c(\mathbf{w})=\mathbf{w}$.
\item if $|\mathbf{w}|> 2L+K$ and $\mathbf{w}= \mathbf{u}_I$ for some integer interval $I$, we have three cases.
\begin{enumerate}
  \item if $0<|\mathbf{u}_I^{\mathrm{suf}}|+|\mathbf{u}_I^{\mathrm{pref}}|\neq L$, then $c(\mathbf{w})=(\mathbf{u}_I^{\mathrm{suf}},\mathbf{u}_I^{\mathrm{pref}})$;
  \item if $I$ is $1$-fitted, that is $|\mathbf{u}_I^{\mathrm{suf}}|=|\mathbf{u}_I^{\mathrm{pref}}|=0$, then $c(\mathbf{w})=c(\zeta^{-1}(\mathbf{u}_I))$;
   \item if $ |\mathbf{u}_I^{\mathrm{suf}}|+|\mathbf{u}_I^{\mathrm{pref}}|=L$, then $c(\mathbf{w})=c(\mathbf{u}_{\tau(I)})$, where
   $$\tau(I)=\left\{\begin{array}{l}
   \mbox{$[m(I),\overline M(I)-1]$ if $|\mathbf{u}_I^{\mathrm{suf}}|\geq r$} \\
                   \mbox{${[\underline m(I),M(I)-1]}$ if $|\mathbf{u}_I^{\mathrm{suf}}|< r$}
                          \end{array}\right..$$
  \end{enumerate}
Here $\overline M(I)$ is the successor of $M(I)$ and $\underline m(I)$ is the predecessor of $m(I)$ in the sequence of $1$-cutting bars of $\zeta$.
\end{enumerate}


\begin{lem}
  The finite coloring $c$ is well-defined.
\end{lem}
\begin{proof}
  We have to check that, for $|\mathbf{w}|>2L+K$, the definition of $c(\mathbf{w})$ does not depend on $I$. This is a direct consequence of our recognizability condition (cf. Lemma \ref{presuf}) when $0<|\mathbf{u}_I^{\mathrm{suf}}|+|\mathbf{u}_I^{\mathrm{pref}}|\neq L$.

 If   $|\mathbf{u}_I^{\mathrm{suf}}|=|\mathbf{u}_I^{\mathrm{pref}}|=0$, that is $I=[m(I),M(I)-1]$, we have  $$\mathbf{u}_I=\zeta(\alpha_{p_I+1})\ldots\zeta(\alpha_{q_I}) $$ and
$\zeta^{-1}(\mathbf{u}_I)=\alpha_{p_I+1}\ldots \alpha_{q_I}.$ Take any interval $J$ such that $\mathbf{u}_I=\mathbf{u}_J$. By recognizability, $J$ is also $1$-fitted and
and we have  $\zeta^{-1}(\mathbf{u}_J)=\zeta^{-1}(\mathbf{u}_I)$.

   For the case (c),  we can write  $$\mathbf{u}_{[m(I),\overline{M}(I)-1]}=\mathbf{u}_{[m(I),M(I)-1]}\mathbf{u}_{[M(I),\overline{M}(I)-1]}$$ with $\mathbf{u}_{[M(I),\overline{M}(I)-1]}=\zeta(\alpha)$ for some letter $\alpha\in\mathcal{A}$. If $|\mathbf{u}_I^{\mathrm{suf}}|\geq r$, such letter is uniquely determined by $\mathbf{u}_I^{\mathrm{suf}}$, since $\zeta_r^{\mathrm{pref}}$ is one-to one on letters.
  Hence, if $\mathbf{u}_I=\mathbf{u}_J$, from Lemma \ref{presuf} we conclude that
  $$\mathbf{u}_{[m(I),\overline{M}(I)-1]}=\mathbf{u}_{[m(I),M(I)-1]}\zeta(\alpha)=\mathbf{u}_{[m(J),M(J)-1]}\zeta(\alpha)=\mathbf{u}_{[m(J),\overline{M}(J)-1]}.$$
Similarly, if $|\mathbf{u}_I^{\mathrm{suf}}|< r$ and $\mathbf{u}_I=\mathbf{u}_J$ we have  $\mathbf{u}_{[\underline{m}(I),M(I)-1]}=\mathbf{u}_{[\underline{m}(J),M(J)-1]}$. Hence the definition of $c(\mathbf{w})$ does not depend on $I$.
\end{proof}

Next we prove that the finite coloring $c$ avoids uniform monochromatic $k$-powers for all $k$ sufficiently large. Suppose that, contrary to our claim, there exists a  uniform monochromatic $k$-power $\mathbf{w}_1\mathbf{w}_2\ldots \mathbf{w}_k$, with $|\mathbf{w}_i|\leq 2L+K$, for each $k$. By definition of $c$, we have in this case $c(\mathbf{w}_i)= c(\mathbf{w}_j)$ if and only if $\mathbf{w}_i= \mathbf{w}_j$. Hence  $\mathbf{w}_1\mathbf{w}_2\ldots \mathbf{w}_k=\mathbf{w}_1^k$.
Since the subset of factors $\mathbf{w}_1\in\mathcal{L}(\mathbf{u})$ with $\mathbf{w}_1\leq 2L+K$ is finite, there exists $\mathbf{w}\in \mathcal{L}(\mathbf{u})$ such  that
 $\mathbf{w}^k$ occurs in $\mathbf{u}$ for all $k>0$. But this means that the periodic word $\mathbf{w}^\infty=\mathbf{w}\mathbf{w}\ldots$ is in the closure of the $\sigma$-orbit of $\mathbf{u}$, which contradicts Proposition \ref{aperiodic}.

Assume now that $\mathbf{w}_1\mathbf{w}_2\ldots \mathbf{w}_k$ is a uniform  $k$-power of $\mathbf{u}$ with $|\mathbf{w}_i|>2L+K$. Set $\mathbf{w}_i=\mathbf{u}_{I_i}$, with $I_1,I_2,\ldots,I_k$ consecutive intervals. If $|\mathbf{u}_{I_i}|$ is not a multiple of $L$, then, since the substitution $\zeta$ is of constant length $L$, we must have   $|\mathbf{u}_{I_i}^{\mathrm{suf}}|+|\mathbf{u}_{I_i}^{\mathrm{pref}}|\neq L$ for all $i$. In this case, it is clear from the definition of $c$ that $c( \mathbf{w}_i)\neq c( \mathbf{w}_{i+1}) $: as a matter of fact, since $|\mathbf{u}^\mathrm{suf}_{I_{i}}|+|\mathbf{u}^\mathrm{pref}_{I_{i+1}}|= L$ and  $ |\mathbf{u}_{I_i}^{\mathrm{suf}}|+|\mathbf{u}_{I_i}^{\mathrm{pref}}|\neq L$,
   we have $|\mathbf{u}^\mathrm{pref}_{I_{i}}|\neq |\mathbf{u}^\mathrm{pref}_{I_{i+1}}|$. Note that, in this particular case, it is sufficient to take $k=2$.

If $|\mathbf{u}_{I_i}|$ is multiple of $L$, then we have either $|\mathbf{u}_{I_i}^{\mathrm{suf}}|+|\mathbf{u}_{I_i}^{\mathrm{pref}}|= L$ for all $i$ or $|\mathbf{u}_{I_i}^{\mathrm{suf}}|+|\mathbf{u}_{I_i}^{\mathrm{pref}}|= 0$ for all $i$.
In both cases, after a finite number of ``translations" $\tau$ (observe that, for a such uniform $k$-power,
$\mathbf{u}_{\tau(I_1)}\ldots \mathbf{u}_{\tau(I_k)}$ is a  uniform  $k$-power of $\mathbf{u}$)
and ``desubstitutions" $\zeta^{-1}$ we fall in one of the previous cases. Hence we conclude that  $c$ avoids uniform monochromatic $k$-factors for all $k$ sufficiently large.
\end{proof}
\begin{rem}
  Theorem \ref{unifo} holds for admissible substitutions of constant length $L$ on an arbitrary finite alphabet $\mathcal{A}$ if one considers the following additional condition: for some $r\in \{1,\ldots,L\}$, the substitutions $\zeta_{r}^{\mathrm{pref}}$ and $\zeta_{r}^{\mathrm{suf}}$ on $\mathcal{A}$ defined by \eqref{subpre}
are both one-to-one on letters.  The proof is exactly the same.
\end{rem}

\begin{rem}
  If $\zeta$ is admissible, then $\mathcal{L}_\zeta=\mathcal{L}(\mathbf{u})$, where   $\mathbf{u}$ is the fixed point of $\zeta$, and $X_\zeta=\overline{\mathcal{O}(\mathbf{u})}$. Consequently, with respect to the finite coloring $c$ of $\mathcal{L}(\mathbf{u})$ defined above, any $\mathbf{v}\in \overline{\mathcal{O}(\mathbf{u})}$ avoids arbitrarily large  uniform monochromatic $k$-powers.
\end{rem}
\begin{eg}\label{1111}
  The one-to-one substitution $\zeta$ on the alphabet $\{0,1\}$ defined by
  \begin{equation*}\label{exem}
  \zeta(0)=0100,\quad\zeta(1)=1101
  \end{equation*}
   is an admissible substitution of constant length $L=4$ with fixed word
   $$\mathbf{u}= 01001101010001001101\ldots $$
   Since the factors $0100$ and $1101$ can only occur in $1$-fitted intervals,  $\zeta$ is strongly recognizable, with recognizability index $K=0$. Taking $r=1$ and the coloring $c$ of $\mathcal{L}(\mathbf{u})$ as defined in the proof of the previous theorem, we see that $\mathbf{u}$ avoids arbitrarily large uniform monochromatic $k$-powers. More precisely, we can take $k=4$.
   In order to check that  $\mathbf{u}$ avoids  uniform monochromatic $4$-powers, one only has to investigate factors of $\mathbf{u}$ of the form $\mathbf{w}_1\mathbf{w}_2\mathbf{w}_3\mathbf{w}_4$ with $l=|\mathbf{w}_i| \leq 8$. In this case, $\mathbf{w}_1\mathbf{w}_2\mathbf{w}_3\mathbf{w}_4$ is a uniform monochromatic $4$-power if, and only if, $\mathbf{w}_1\mathbf{w}_2\mathbf{w}_3\mathbf{w}_4=\mathbf{w}_1^4$. For example, factors of the form $1^4$ and $0^4$ do not occur in $\mathbf{u}$, and a similar analysis can be done for the remaining  values of $l\leq 8$.
\end{eg}

\section{Avoiding arbitrarily large monochromatic $k$-powers with uniformly bounded gaps.}\label{JP2}

 J. Justin and G. Pirillo \cite{JP} considered the infinite word
  $$\mathbf{u}_J=0000100001000010000101111\ldots,$$
  which is a fixed point of the substitution $\zeta_J$ defined by $\zeta_J(0)=00001$ and $\zeta_J(1)=11110$, and constructed a $3$-coloring of its factors
 which avoids arbitrarily large  monochromatic $k$-powers with uniformly bounded gaps.  In this section we generalize this construction to a certain class of infinite words over a finite alphabet $\mathcal{A}$. Recall that  $\mathbf{w}=\mathbf{w}_1\mathbf{w}_2\ldots \mathbf{w}_k\in\mathcal{A}^*$ is an \emph{abelian $k$-power}  if any two of  the factors $\mathbf{w}_i$ are permutations of each other.

\begin{thm}Let $\mathbf{u}$ be an infinite word over $\mathcal{A}$.
 Assume that each factor of $\mathbf{u}$ occurs in $\mathbf{u}$ with a uniform frequency and that $\mathbf{u}$ avoids abelian $k_0$-powers for some $k_0>0$. Then there exists a  finite coloring $c$ of $\mathcal{L}(\mathbf{u})$ which avoids arbitrarily large  monochromatic $k$-powers with uniformly bounded gaps.
\end{thm}
\begin{proof}
  By hypothesis, any letter $\alpha\in\mathcal{A}$ occurs in $\mathbf{u}$ with uniform frequency, say $d_\alpha>0$. We have $\sum_{\alpha\in\mathcal{A}}d_\alpha=1$. Define the following  finite coloring of $\mathcal{L}(\mathbf{u})$:
  $$c(\mathbf{w})= \{\alpha\in\mathcal{A}\colon\,\, \frac{|\mathbf{w}|_\alpha}{|\mathbf{w}|}>d_\alpha\}.$$
 In particular, if $c(\mathbf{w})=\emptyset$ then $\frac{|\mathbf{w}|_\alpha}{|\mathbf{w}|}=d_\alpha$ for all $\alpha\in\mathcal{A}$. Toward a contradiction, suppose that $c$ admits arbitrarily large  monochromatic $k$-powers with uniformly bounded gaps. Then there exists $p>0$ such that  $\mathbf{u}$ admits a monochromatic $k$-power $\mathbf{w}=\mathbf{w}_1\mathbf{w}_2\ldots \mathbf{w}_k$ with gaps bounded by $p$, for arbitrarily large $k$.

 We claim that, in the monochromatic $k$-power $\mathbf{w}=\mathbf{w}_1\mathbf{w}_2\ldots \mathbf{w}_k$, we must have  $c(\mathbf{w}_i)=\emptyset$ for all $i$. As a matter of fact, otherwise we would have, for some letter $\alpha$, $|\mathbf{w}_i|_\alpha >d_\alpha|\mathbf{w}_i|$ for all $i$ and arbitrarily large $k$. Since we are suposing that $|\mathbf{w}_i|<p$, then there would exist some $\epsilon_0>0$ such that $|\mathbf{w}_i|_\alpha > d_\alpha|\mathbf{w}_i|+\epsilon_0$ for all $i$ and arbitrarily large $k$. Consequently,
$$\frac{|\mathbf{w}_i|_\alpha}{|\mathbf{w}_i|}>d_\alpha+\frac{\epsilon_0}{p},$$
which contradicts the hypothesis on the uniformity of letter frequencies.

Therefore we must have $ c(\mathbf{w}_i)=\emptyset$ for all $i$, and we can mimic the argument used by J. Justin and G. Pirillo (Theorem 2, \cite{JP}) in order to conclude our proof. For completeness, next we present the details.
 Consider the alphabet $\mathcal{A}_p=\{\mathbf{v}\in\mathcal{A}^*\colon\,|\mathbf{v}|<p\}$ and the morphism $\varphi:\mathcal{A}_p^*\to \mathbb{N}$ given by $\varphi(\mathbf{v})=|\mathbf{v}|$, where the length of $\mathbf{v}$ is taken with respect to the alphabet $\mathcal{A}$. By Theorem 4.2.1 of \cite{Pi}, $\varphi$ is \emph{repetitive}. Recall that, given an alphabet $\tilde{\mathcal{A}}$, a mapping  $\tilde \varphi:\tilde{\mathcal{A}}^*\to S$ to a set $S$ is repetitive if for each $k_0$, there exists an integer $k$, such that each word $\mathbf{v}\in \tilde{\mathcal{A}}^*$ of length $k$ contains a factor of the form $\mathbf{v}_1\ldots \mathbf{v}_{k_0}$, with $\mathbf{v}_i\in \tilde{\mathcal{A}}^*$ and $\tilde\varphi(\mathbf{v}_1)=\cdots=\tilde\varphi(\mathbf{v}_{k_0})$. In our setting, this means that for $k$ sufficiently large, there exist integers $i_1,i_2,\ldots,i_{k_0},i_{k_0+1}$
 such that the  consecutive factors $\mathbf{v}_1 $, $\mathbf{v}_2 $ and $ \mathbf{v}_3$ of  $\mathbf{w}=\mathbf{w}_1\mathbf{w}_2\ldots \mathbf{w}_k$  defined by $$\mathbf{v}_1=\mathbf{w}_{i_1}\mathbf{w}_{i_1+1}\ldots\mathbf{w}_{i_2-1},\,\mathbf{v}_2=\mathbf{w}_{i_2}\mathbf{w}_{i_2+1}\ldots\mathbf{w}_{i_3-1},\, \ldots,\,\mathbf{v}_{k_0}=\mathbf{w}_{i_{k_0}}\mathbf{w}_{i_{k_0}+1}\ldots \mathbf{w}_{i_{k_0+1}-1}$$ have the same length $L=|\mathbf{v}_1|=\ldots =|\mathbf{v}_{k_0}|$.
 On the other hand, since  $ c(\mathbf{w}_i)=\emptyset$, we have $|\mathbf{w}_i|_\alpha=d_\alpha|\mathbf{w}_i|$ for all $\alpha\in\mathcal{A}$. This fact implies that
 $|\mathbf{v}_i|_\alpha=d_\alpha |\mathbf{v}_i|=d_\alpha L$, for all $i$.
 Hence any two of the factors $\mathbf{v}_i$  are permutations of each other, that is  $\mathbf{v}_1\mathbf{v}_2\ldots \mathbf{v}_{k_0}$ is an abelian $k_0$-power, in contradiction with our hypothesis.

\end{proof}
\begin{eg}\label{abc} Consider the substitution $\zeta$ on three letters defined by
$$\zeta(a)=aabc,\,\, \zeta(b)=bbc,\,\, \zeta(c)=acc.$$ This is a primitive substitution whose fixed word
$$\mathbf{u}=aabcaabcbbcaccaabcaabcbbcaccbbcbbc\ldots $$ avoids abelian 3-powers \cite{Dek}. By Proposition  \ref{primitive},
each factor of $\mathbf{u}$ occurs in $\mathbf{u}$ with a uniform frequency.
Hence, there exists a finite coloring of $\mathcal{L}(\mathbf{u})$ which avoids arbitrarily large  monochromatic $k$-powers with uniformly bounded gaps.
\end{eg}
\section{Factorizations of infinite words}\label{LZ}

Let  $\mathbf{u}$ be an infinite word over a finite alphabet $\mathcal{A}$, and  $c$ be any finite coloring of $\mathcal{L}(\mathbf{u})$.
A known  consequence of  Ramsey's Theorem (for infinite complete graphs) is the following.
\begin{thm}\cite{S}
 There exists a factorization $\mathbf{u}=\mathbf{v} \mathbf{u}_1 \mathbf{u}_2 \mathbf{u}_3\ldots $ with $c(\mathbf{u}_1 )=c(\mathbf{u}_2 )=c(\mathbf{u}_3 )=\ldots$.
\end{thm}
Inspired by this result, T. Brown \cite{Br3} and L. Zamboni \cite{Za} posed independently the following question: \emph{Given an infinite word over the alphabet $\mathcal{A}$, does there exist a finite coloring of $\mathcal{L}(\mathbf{u})$ which avoids monochromatic factorizations of $\mathbf{u}$?}  Let $\mathcal{P}$ denote the set of all infinite words over a finite alphabet for which this question has a positive answer. That is, for each $\mathbf{u}\in\mathcal{P}$ there exists a finite coloring $c$ of $\mathcal{L}(\mathbf{u})$ satisfying the property
that for each factorization $\mathbf{u}=\mathbf{u}_1 \mathbf{u}_2 \mathbf{u}_3\ldots $ there exist $i, j$ for which $c(\mathbf{u}_i)\neq c(\mathbf{u}_j)$.
 In \cite{LPZ} the authors proved that all non-uniformly recurrent words and various classes of non-periodic uniformly recurrent words are in $\mathcal{P}$: \emph{non-periodic balanced} words, and all words $\mathbf{u}\in\mathcal{A}^\mathbb{N}$ satisfying $\lambda_{\mathbf{u}}(n+1)-\lambda_{\mathbf{u}}(n)=1$ for all sufficiently large $n$, where $\lambda_{\mathbf{u}}(n)$ denotes the number of distinct factors of $\mathbf{u}$ of length $n$. Recall that an infinite word $\mathbf{u}$ is balanced if for any two finite factors $\mathbf{w}_1$ and $\mathbf{w}_2$ of $\mathbf{u}$ with the same length, we have
 $||\mathbf{w}_1|_\alpha-|\mathbf{w}_2|_\alpha|\leq 1$ for any letter $\alpha\in \mathcal{A}$.
  Next we prove that all substitutive words associated to admissible substitutions on a finite alphabet avoiding arbitrarily large $k$-powers  also belong to $\mathcal{P}$.
\begin{thm}\label{abap}
  Let $\zeta$ be an admissible substitution on a finite alphabet $\mathcal{A}$, with recognizability index $K$. Let $\mathbf{u}$ be a fixed point of  $\zeta$. Then  $\mathbf{u}\in \mathcal{P}$.
\end{thm}
\begin{proof}
Let $L_k=\max\{|\zeta^k(\alpha)|\colon\,\alpha\in \mathcal{A}\}$. For $k=1$ we also denote $L=L_1$.
In view of Proposition \ref{aperiodic}, we can fix $k_0$ such that  no word of the form $\mathbf{w}^{k_0}$ with $|\mathbf{w}|\leq 2L+K$ occurs in $\mathbf{u}$.
 Fix an integer $q$ satisfying $|\zeta^q(\alpha)|>k_0(K+2L)$ for all $\alpha\in\mathcal{A}$, and fix another integer $P>q$ satisfying $|\zeta^P(\alpha)|>2L_q+K_q$ for all $\alpha\in\mathcal{A}$, where $K_q$ is the recognizability index of $\zeta^q$ (cf. Proposition \ref{pcuttingbars}).

Define a finite coloring $c$ of $\mathcal{L}(\mathbf{u})$ as follows. Observe first that if $\mathbf{v}$ is a prefix of $\mathbf{u}$, then $\mathbf{v}$ is also a prefix of  $\zeta^k(\mathbf{v})$, since   $\mathbf{u}$ is a fixed point of $\zeta$.  For  $\mathbf{w}\in\mathcal{L}(\mathbf{u})$:
\begin{enumerate}
\item if $\mathbf{w}$ is not a prefix of $\mathbf{u}$, then $c(\mathbf{w})=0$.
\item if $\mathbf{w}$ is a prefix of $\mathbf{u}$, we have three cases.
\begin{enumerate}
  \item if $|\mathbf{w}|\leq 2L+ K$, then $c(\mathbf{w})=\mathbf{w}$;
  \item if $\mathbf{w}=\zeta^k(\mathbf{v})$ for some $k>0$ and some prefix $\mathbf{v}$ of $\mathbf{u}$ satisfying $|\mathbf{v}| \leq 2L+ K < |\zeta(\mathbf{v})|$,  then $c(\mathbf{w})=(\mathbf v,k\mod P)$;
   \item otherwise $c(\mathbf{w})=1$.
                          \end{enumerate}
\end{enumerate}

We claim that $\mathbf{u}$ avoids monochromatic factorizations with respect to $c$. Suppose, on the contrary, that  $\mathbf{u}=\mathbf{u}_1 \mathbf{u}_2 \mathbf{u}_3\ldots$ is a monochromatic factorization. Since $\mathbf{u}_1$ is a prefix of $\mathbf{u}$, we must have $c(\mathbf{u}_1)\neq 0$, and consequently $c(\mathbf{u}_i)\neq 0$ for all $i$. On the other hand, if $|\mathbf{u}_i|<2L+K$ for some $i$, then, by definition of $c$, we would have $\mathbf{w}=\mathbf{u}_i$ for all $i$ and consequently $\mathbf{u}=\mathbf{w}^\infty$, in contradiction with Proposition \ref{aperiodic}.

So we can suppose that each $\mathbf{u}_i$ is a prefix of $\mathbf{u}$ and $|\mathbf{u}_i|> 2L+K$. Write
\begin{equation*}\label{us}
\mathbf{u}_1=\mathbf{u}_{[1,|\mathbf{u}_1|]},\,\,\mathbf{u}_2=\mathbf{u}_{[|\mathbf{u}_1|+1,|\mathbf{u}_1|+|\mathbf{u}_2|]}=\mathbf{u}_{[1,|\mathbf{u}_2|]},\,
\mathbf{u}_3=\mathbf{u}_{[|\mathbf{u}_1|+|\mathbf{u}_2|+1,|\mathbf{u}_1|+|\mathbf{u}_2|+|\mathbf{u}_3|]}=\mathbf{u}_{[1,|\mathbf{u}_3|]}, \,\ldots
\end{equation*}
By recognizability (cf. Lemma \ref{presuf}), this means that $\mathbf{u}_i^\mathrm{pref}=\emptyset$ for each $i$, which implies that  each factor $\mathbf{u}_i$ is $1$-fitted.

Assume that $c(\mathbf{u}_i)=1$ for all $i$.   Take $r_i>  0$ such that the factor $\zeta^{-r_i}(\mathbf{u}_i)$ is not $1$-fitted. Clearly  $|\zeta^{-r_i}(\mathbf{u}_i)|> 2L+K$ (otherwise $c(\mathbf{u}_i)\neq 1$) and $c(\zeta^{-r_i}(\mathbf{u}_i))=1$. Take $r=\min_i\{r_i\}$. The factorization $\mathbf{u}=\zeta^{-r}(\mathbf{u}_1) \zeta^{-r}(\mathbf{u}_2) \zeta^{-r}(\mathbf{u}_3)\ldots$ is also monochromatic (all factors with color $1$), and some factor  $\zeta^{-r}(\mathbf{u}_i)$ is not $1$-fitted, which, as we have seen, is impossible.

 So we must have $c(\mathbf{u}_i)=(\mathbf{w},k)$ for some prefix $\mathbf{w}$, with $|\mathbf{w}|\leq 2L+K<|\zeta(\mathbf{w})|$, and $0\leq k<P$. In this case the factorization is of the form
\begin{equation}\label{prod}
\mathbf{u}=\prod_{i\geq 1}(\zeta^{k+n_iP}(\mathbf{w}))^{s_i},\end{equation}
with  $n_i\geq 0$.

We claim that the set $\{i\colon\, \,n_{i+1}>n_i\}$ is nonempty. Towards a contradiction, suppose that $n_{i+1}\leq n_i$ for all $i$. In this case, there exists $i_0$ such that $n_i=n_{i_0}$ for all $i\geq i_0$, and, in view of \eqref{prod}, we can write $\mathbf{u}=\mathbf{v}(\zeta^{k+n_{i_0}P}(\mathbf{w}))^{\infty}$, for some prefix $\mathbf{v}$ of $\mathbf{u}$, which contradicts Proposition \ref{aperiodic}.

Let $$i_0=\min\{i:\, n_{i+1}>n_i\}.$$ For simplicity of exposition we assume that $i_0=1$, but the argument below holds for any other possible value of $i_0$.
Applying $\zeta^{-(k+n_{1}P)}$ to \eqref{prod}, we obtain
$$\mathbf{u}=\mathbf{w}^{s_1}(\zeta^{(n_{2}-n_{1})P}(\mathbf{w}))^{s_{2}}\ldots.$$

Now, since $\zeta^{(n_2-n_1)P}(\mathbf{w})$ is a prefix of the fixed point $\mathbf{u}$ and $P\leq (n_2-n_1)P$, we have $$\zeta^{(n_2-n_1)P}(\mathbf{w})=\mathbf{u}_{[1,|\zeta^{(n_2-n_1)P}(\mathbf{w})|]}\,\,\,\mbox{and}\,\,\,|\zeta^{(n_2-n_1)P}(\mathbf{w})|\geq|\zeta^{P}(\mathbf{w})|>2L_q+K_q.$$ Hence, by recognizability (see Lemma \ref{pcuttingbars}), we see that $|\mathbf{w}^{s_1}|+1$ is a $q$-cutting bar. But
$$1\leq |\mathbf{w}^{s_1}|=s_1|\mathbf{w}|\leq k_0(K+2L)<|\zeta^q(\alpha)|,$$
 for all $\alpha\in\mathcal{A}$, which means that $|\mathbf{w}^{s_1}|+1\notin E_q$. This is a contradiction, and consequently  we conclude that $\mathbf{u}$ avoids monochromatic factorizations with respect to $c$.
\end{proof}

\begin{eg}
  Consider the primitive substitution $\zeta$ on the alphabet $\{0,1\}$ defined by $\zeta(0)=0100$ and $\zeta(1)=1101,$ as in Example \ref{1111}. Let $\mathbf{u}$ be a fixed point of $\zeta$.  This is a uniformly recurrent, non-periodic and non-balanced infinite word. As a matter of fact:  uniform recurrence and non-periodicity is a consequence of $\zeta$ being admissible; one can easily check that
  $|\zeta^n(1)|_1-|\zeta^n(0)|_1=2^n,$ whereas $|\zeta^n(1)|=|\zeta^n(0)|=4^n$, which means that $\mathbf{u}$ is  non-balanced. On the other hand, we have  
             $$\lambda_{\mathbf{u}}(2)-\lambda_{\mathbf{u}}(1)=2,\,\lambda_{\mathbf{u}}(3)-\lambda_{\mathbf{u}}(2)=4,\, \lambda_{\mathbf{u}}(4)-\lambda_{\mathbf{u}}(3)=5,\ldots.$$
            But, by Theorem 3 of \cite{Mo} (which holds only for substitutions of constant length), we know that the sequence  $\lambda_{\mathbf{u}}(n+1)-\lambda_{\mathbf{u}}(n)$ takes values in  a finite set and each of these values occurs for infinitely many $n$. Hence  the infinite word $\mathbf{u}$ certainly does not satisfies    $\lambda_{\mathbf{u}}(n+1)-\lambda_{\mathbf{u}}(n)=1$ for all sufficiently large $n$.
              We conclude, by Theorem \ref{abap}, that $\mathbf{u}$ provides an example of a non-periodic infinite word in $\mathcal{P}$ which is not considered in the paper of de Luca, Pribavkina, and Zamboni \cite{LPZ}.\end{eg}

\begin{eg}  Consider the substitution $\zeta$ on the alphabet $\{a,b,c\}$ defined by $\zeta(a)=aabc$, $\zeta(b)=bbc$, and $\zeta(c)=acc$, as in the Example \ref{abc}. This is an admissible substitution, and consequently $\mathbf{u}$ is non-periodic. Since $aabc$, $bbc$ and $acc$ only occur in $1$-fitted intervals, $\zeta$ is strongly recognizable.  Hence $\mathbf{u}\in\mathcal{P}$. It is known that $\mathbf{u}$  avoids  abelian $3$-powers \cite{Dek}.

\end{eg}

\begin{rem}
In \cite{LZ}, the authors considered the class $\mathcal{P}_1$ of all infinite words over finite alphabets admitting a prefixal factorization. They  associated to each $\mathbf{u}\in\mathcal{P}_1$ a ``derived" infinite word $\delta(\mathbf{u})$, which may or may not belong to $\mathcal{P}_1$, and defined the class $\mathcal{P}_\infty$ of all words $\mathbf{u}$ in $\mathcal{P}_1$ such that $\delta^n(\mathbf{u})\in\mathcal{P}_1$ for all $n\geq 1$.
 de Luca and Zamboni \cite{LZ} showed that any word $\mathbf{u}$ which does not belong to $\mathcal{P}_\infty$ admits a finite coloring of its factors which avoids monochromatic factorizations of $\mathbf{u}$, that is $\mathbf{u}\in \mathcal{P}$.

 Recall that the Fibonacci word and the Tribonacci word, respectively, are fixed points of the admissible substitutions $\zeta_f(0)=01,\zeta_f(1)=0$ and $\zeta_t(1)=12,\zeta_t(2)=13,\zeta_t(3)=1$, respectively. Consequently, taking Theorem \ref{abap} into account, we see that $\mathbf{u}\in \mathcal{P}$. On the other hand,
the Fibonacci word and the Tribonacci word belong to  $\mathcal{P}_\infty$ \cite{LZ}. Hence, Theorem \ref{abap} does not follow from de Luca and Zamboni results.
\end{rem}


\begin{thebibliography}{99}
%
%


%

\bibitem{Br3} \newblock T.  Brown,
\newblock   \emph{Colorings of the factors of a word},
\newblock preprint, Department of Mathematics, Simon
Fraser University, Canada (2006).


\bibitem{Dek}
\newblock F. M. Dekkiing,
   \newblock\emph{Strongly Non-Repetitive Sequences and Progression-Free Sets},
  \newblock J. of Combin. Theory, Series A \textbf{27}, (1979), pp. 181 -- 185.






\bibitem{Host}
\newblock B. Host,
   \newblock\emph{ Valeurs propres des syst\`{e}mes dynamiques d\'{e}finis par des substitutions de longueur variable},
  \newblock  Ergodic Theory Dynam. Systems 6 (1986), no. 4, 529--540.


\bibitem{LZ}
\newblock A. de Luca,  L. Zamboni,
\newblock \emph{On prefixal factorizations of words},
\newblock  European J. Combin. 52 (2016), A, 59--73.



\bibitem{LPZ} \newblock A. de Luca, E. V. Pribavkina, L. Zamboni,
\newblock \emph{A Coloring Problem for Infinite Words},
\newblock  J. Combin. Theory Ser. A 125 (2014), 306--332.

\bibitem{Ju}
\newblock J. Justin,
   \newblock\emph{ Characterisation of repetitive commutative semigroups  },
  \newblock J. Algebra, \textbf{1} (1972), pp. 87 -- 90.



\bibitem{JP}
\newblock J. Justin and G. Pirillo,
   \newblock\emph{Two combinatorial properties of partitions of the free semigroup into finitely many parts},
  \newblock Discrete Mathematics, \textbf{52}, 2-3, (1984), pp. 299 -- 303.




\bibitem{Mo}
\newblock B. Moss\'{e},
   \newblock\emph{ Reconnaissabilit\'{e} des substitutions et complxit\'{e} des suites automatiques},
  \newblock Bull. Soc. math France, \textbf{124}, (1996), pp. 329 -- 346.


\bibitem{Pi}
\newblock G. Pirillo,
   \newblock \emph{Repetitive mappings and morphisms},
  \newblock M. Lothaire (Ed.), Combinatorics on Words, Addison-Wesley, London (1983), pp. 55--62.

\bibitem{Q} \newblock M. Queffelec,
\newblock {Substitution dynamical systems -- spectral analysis,}
\newblock Lecture Notes in Math. 1294, Springer-Verlag, 1987.


\bibitem{ST}
\newblock V. Salo, I. T\"{o}rm\"{a},
   \newblock \emph{Factor Colorings of Linearly Recurrent Words},
  \newblock arXiv:1504.05821 [math.CO].



\bibitem{S}
\newblock M.P. Sch\"{u}tzenberger,
   \newblock Quelques probl\`{e}mes combinatoires de la th\'{e}orie des automates,
  \newblock Cours profess\'{e} \`{a} l'Institut de Programmation, Fac. Sciences de Paris (1966--1967) r\'{e}dig\'{e} par J.-F. Perrot.

\bibitem{Za}
     \newblock L.Q. Zamboni,
     \newblock \emph{ A Note on Coloring Factors of Words},
     \newblock  in Oberwolfach Report 37/2010, Mini-workshop,  Combinatorics on Words, August 22-27, 2010, pp. 42-- 44.






%
%
%
%
%
%
%
%
%
%
%
%

\end{thebibliography}
\end{document}